\documentclass[12pt]{amsart}
\usepackage{a4wide}
\usepackage{amsmath,amssymb,amsthm}
\usepackage{amsfonts,amsthm,amsmath,amssymb,amscd,color}
\usepackage{graphicx}
\usepackage{ulem}
\usepackage{enumerate}
\usepackage{parskip}
\usepackage{graphicx}
\usepackage{verbatim}

\newtheorem{corollary}{Corollary}[section]
\newtheorem{remark}{Remark}[section]
\newtheorem{lemma}{Lemma}[section]
\newtheorem{proposition}{Proposition}[section]
\newtheorem{theorem}{Theorem}[section]
\newtheorem{counterexample}{Counterexample}[section]

\makeatletter
\@addtoreset{equation}{section}

\newcommand{\bes}{\begin{displaymath}}
\newcommand{\ees}{\end{displaymath}}
\newcommand{\be}{\begin{equation}}
\newcommand{\ee}{\end{equation}}
\newcommand{\ba}{\begin{eqnarray}}
\newcommand{\ea}{\end{eqnarray}}
\newcommand{\bas}{\begin{eqnarray*}}
\newcommand{\eas}{\end{eqnarray*}}
\newcommand{\@Bbb}[1]{\ensuremath{\Bbb #1}}

\newcommand{\B}{{\@Bbb B}}
\newcommand{\C}{{\@Bbb C}}
\newcommand{\E}{{\@Bbb E}}
\newcommand{\F}{{\@Bbb F}}
\newcommand{\G}{{\@Bbb G}}
\renewcommand{\P}{{\@Bbb P}}

\newcommand{\Q}{{\@Bbb Q}}
\newcommand{\bQ}{{\@Bbb Q}}
\newcommand{\N}{{\@Bbb N}}
\newcommand{\R}{{\@Bbb R}}
\newcommand{\T}{{\@Bbb T}}
\newcommand{\bbR}{{\@Bbb R}}
\newcommand{\W}{{\@Bbb W}}
\newcommand{\Z}{{\@Bbb Z}}
\newcommand{\bbZ}{{\@Bbb Z}}

\newcommand{\@s}[1]{\ensuremath{\mathcal #1}}
\newcommand{\cA}{\@s A}
\newcommand{\cB}{\@s B}
\newcommand{\cC}{\@s C}
\newcommand{\cD}{\@s D}
\newcommand{\cE}{\@s E}
\newcommand{\cF}{\@s F}
\newcommand{\cG}{\@s G}
\newcommand{\cH}{\@s H}
\newcommand{\cI}{\@s I}
\newcommand{\cJ}{\@s J}

\newcommand{\cK}{\@s K}
\newcommand{\cL}{\@s L}
\newcommand{\cN}{\@s N}
\newcommand{\cM}{\@s M}
\newcommand{\cO}{\@s O}
\newcommand{\cP}{\@s P}
\newcommand{\cQ}{\@s Q}
\newcommand{\cR}{\@s R}
\newcommand{\cS}{\@s S}
\newcommand{\cT}{\@s T}
\newcommand{\cU}{\@s U}
\newcommand{\cV}{\@s V}
\newcommand{\cW}{\@s W}
\newcommand{\cX}{\@s X}
\newcommand{\cY}{\@s Y}
\newcommand{\cZ}{\@s Z}

\newcommand{\@bm}[1]{\ensuremath{\mathbf #1}}
\newcommand{\bma}{\@bm a}\newcommand{\bmA}{\@bm A}
\newcommand{\bmb}{\@bm b}\newcommand{\bmB}{\@bm B}
\newcommand{\bmc}{\@bm c}\newcommand{\bmC}{\@bm C}
\newcommand{\bmd}{\@bm d}\newcommand{\bmD}{\@bm D}
\newcommand{\bme}{\@bm e}
\newcommand{\bmf}{\@bm f}\newcommand{\bmF}{\@bm F}
\newcommand{\bmg}{\@bm g}\newcommand{\bmG}{\@bm G}
\newcommand{\bmh}{\@bm h}\newcommand{\bmH}{\@bm H}
\newcommand{\bmi}{\@bm i}\newcommand{\bmI}{\@bm I}
\newcommand{\bmj}{\@bm j}
\newcommand{\bmk}{\@bm k}\newcommand{\bmK}{\@bm K}
\newcommand{\bml}{\@bm l}
\newcommand{\bmm}{\@bm m}\newcommand{\bmM}{\@bm M}
\newcommand{\bmn}{\@bm n}
\newcommand{\bmo}{\@bm o}
\newcommand{\bmp}{\@bm p}
\newcommand{\bmq}{\@bm q}\newcommand{\bmQ}{\@bm Q}
\newcommand{\bmr}{\@bm r}
\newcommand{\bms}{\@bm s}\newcommand{\bmS}{\@bm S}
\newcommand{\bmt}{\@bm t}
\newcommand{\bmu}{\@bm u}\newcommand{\bmU}{\@bm U}
\newcommand{\bmw}{\@bm w}\newcommand{\bmW}{\@bm W}
\newcommand{\bmv}{\@bm v}\newcommand{\bmV}{\@bm V}
\newcommand{\bmx}{\@bm x}\newcommand{\bmX}{\@bm X}\newcommand{\bx}{\@bm x}
\newcommand{\bmy}{\@bm y}\newcommand{\bmY}{\@bm Y}\newcommand{\by}{\@bm y}
\newcommand{\bmz}{\@bm z}\newcommand{\bmZ}{\@bm Z}

\newcommand{\bmzero}{\@bm 0}

\newcommand{\supp}{\mathop{\rm supp}}
\newcommand{\dist}{\mathop{\rm dist}}

\newcommand{\@g}[1]{\ensuremath{\mathfrak #1}}
\newcommand{\gA}{\@g A}
\newcommand{\gD}{\@g D}
\newcommand{\gJ}{\@g J}
\newcommand{\gF}{\@g F}
\newcommand{\gM}{\@g M}
\newcommand{\gR}{\@g R}

\newcommand{\Lip}{\mathop{\mathrm{Lip}}}

\newcommand{\mathbbm}[1]{{#1\!\!#1}}
\newcommand{\ind}{{\mathbbm{1}}}

\newcommand{\commentout}[1]{{}}

\newcommand{\RR}{\mathbb{R}}
\newcommand{\QQ}{\mathbb{Q}}

\newcommand{\NN}{\mathbb{N}}

\newcommand{\CDD}{\mathcal{D}}
\newcommand{\CE}{\mathcal{E}}
\newcommand{\CF}{\mathcal{F}}
\newcommand{\CK}{\mathcal{K}}
\newcommand{\CM}{\mathcal{M}}

\newcommand{\CP}{\mathcal{P}}

\newcommand{\eps}{\varepsilon}

\newcommand{\BL}{{\mathrm{BL}}}
\newcommand{\TV}{{\mathrm{TV}}}
\newcommand{\FM}{{\mathrm{FM}}}

\newcommand{\pair}[2]{\left\langle #1 , #2 \right\rangle}

\makeatother
\begin{document}

\title[Schur-like property for measures and consequences]{On a Schur-like property for spaces of measures and its consequences}
\author{Sander C. Hille}
\address{Mathematical Institute, Leiden University, P.O. Box 9512, 2300 RA Leiden, The Netherlands, (SH, MZ)}
\email{\{shille,m.a.ziemlanska\}@math.leidenuniv.nl}
\author{Tomasz Szarek}
\address{Institute of Mathematics, University of Gda\'nsk, Wita Stwosza 57, 80-952 Gda\'nsk, Poland (TS)}
\email{szarek@intertele.pl}
\author{Daniel T.H. Worm}
\address{TNO, Cyber security and Robustness, P.O. Box 96800, 2509 JE, The Hague, The Netherlands, (DW)}
\email{daniel.worm@tno.nl}
\author{Maria A. Ziemla\'nska}

\date{{\small 30th March 2017}}

\thanks{The work of MZ has been partially supported by a Huygens Fellowship of Leiden University.
The research of TS has been partially supported by the
  Polish NCN grants 2016/21/B/ST1/00033.}
\maketitle

\begin{abstract}
A Banach space has the Schur property when every weakly convergent sequence converges in norm. We prove a Schur-like property for measures:  if a sequence of finite signed Borel measures on a Polish space is such that it is bounded in total variation norm and such that for each bounded Lipschitz function the sequence of integrals of this function with respect to these measures converges, then the sequence converges in dual bounded Lipschitz norm or Fortet-Mourier norm to a measure. Two main consequences result: the first is equivalence of concepts of equicontinuity in the theory of Markov operators in probability theory and the second concerns conditions for the coincidence of weak and norm topologies on sets of measures that are bounded in total variation norm that satisfy additional properties. Finally, we derive weak sequential completeness of the space of signed Borel measures on Polish spaces from the Schur-like property.
\end{abstract}

\section{Introduction}

The mathematical study of dynamical systems in discrete or continuous time on spaces of probability measures has a long-lasting history in probability theory (as Markov operators and Markov semigroups, see e.g. \cite{Meyn-Tweedie:2009}) and the field of Iterated Function Systems \cite{Barnsley_ea:1988, Lasota-Yorke:1994} in particular. In analysis there is a growing interest in solutions to evolution equations in spaces of positive or signed measures, e.g. in the study of structured population models \cite{Ackleh-Ito:2005, Canizo_ea:2013, Carillo_ea:2012}, crowd dynamics \cite{Piccoli-Tosin:2011} or interacting particle systems \cite{Evers_ea:2016}. Although an extensive body of functional analytic results have been obtained within probability theory (e.g. see \cite{Billingsley:1968, Bogachev-II:2007, Dudley:1966, LeCam:1957}), there is still need for further results, driven for example by the topic of evolution equations in space of measures, in which there is no conservation of mass.

This paper provides such functional analytic results in two directions: one concerning properties of families of Markov operators on the space of finite signed Borel measures $\CM(S)$ on a Polish space $S$ that satisfy equicontinuity conditions (Theorem \ref{thrm:equicvalence equicont}). The other provides conditions on subsets of $\CM(S)$, where $S$ is a Polish space, such that weak topology on $\CM(S)$ coincides with the norm topology defined by the Fortet-Mourier or dual bounded Lipschitz norm $\|\cdot\|_\BL^*$ (Theorem \ref{thrm:coincidence topologies} and similar results in Section \ref{sec:coincidence topologies}). 

Both are built on Theorem \ref{thrm:weak-implies-strong convergence}, which states that if a sequence of signed measures is bounded in total variation norm and has the property that all real sequences are convergent that result from pairing the given sequence of measures by means of integration to each function in the space of bounded Lipschitz functions, $\BL(S)$, then the sequence is convergent for the $\|\cdot\|_\BL^*$-norm. This is a Schur-{\it like} property. Recall that a Banach space $X$ has the Schur property if every weakly convergent sequence in $X$ is norm convergent (e.g. \cite{Albiac-Kalton:2006}, Definition 2.3.4). For example, the sequence space $\ell^1$ has the Schur property (cf. \cite{Albiac-Kalton:2006}, Theorem 2.3.6). Although the dual space of $(\CM(S),\|\cdot\|_\BL^*)$ is isometrically isomorphic to $\BL(S)$ (cf. \cite{Hille-Worm:2009}, Theorem 3.7), the (completion of the) space $(\CM(S),\|\cdot\|_\BL^*)$ is not a Schur space, generally (see Counterexample \ref{counter:Schur prop}). The condition of bounded total variation cannot be omitted.

Properties of the space of Borel probability measures on $S$ for the weak topology induced by pairing with $C_b(S)$ have been widely studied in probability theory, e.g. consult \cite{Bogachev-II:2007} for an overview.  Dudley \cite{Dudley:1966} studied the pairing between {\it signed measures} and the space of bounded Lipschitz functions, $\BL(S)$, in further detail. Pachl investigated extensively the related pairing with $\mathrm{U}_b(S)$, the space of uniformly continuous and bounded functions \cite{Pachl:1979,Pachl:2013}. See also \cite{Kalton:2004}. Because of our interest in equicontinuous families of Markov operators on the one hand, which is intimately tied to `test functions' in the space $\BL(S)$, and to dynamical systems in spaces of measures equipped with the $\|\cdot\|_\BL^*$-norm, or flat metric, on the other hand, we consider novel functional analytic properties of the space of finite signed Borel measures $\CM(S)$ for the $\BL(S)$-weak topology in relation to the $\|\cdot\|_\BL^*$-norm topology. 

Equicontinuous families of Markov operators were introduced in relation to asymptotic stability: the convergence of the law of stochastic Markov process to an invariant measure (e.g. e-chains \cite{Meyn-Tweedie:2009}, e-property \cite{Czapla-Horbacz:2014, KPS, Lasota-Szarek:2006, SSU:2010}, Cesaro-e-property \cite{Worm-thesis:2010}, Ch.7; see also \cite{Jamison:1964}). Hairer and Mattingly introduced the so-called asymptotic strong Feller property for that purpose \cite{Hairer-Mattingly:2006}. Theorem \ref{thrm:equicvalence equicont} rigorously connects two dual viewpoints -- concerning equicontinuity: Markov operators acting on measures (laws) and Markov operators acting on functions (observables). In dynamical systems theory too, there is special interest in ergodicity properties of maps with equicontinuity properties (e.g. \cite{Li_ea:2015}).

The structure of the paper is as follows. After having introduced some notation and concepts in Section \ref{sec:preliminaries} we provide in  Section \ref{sec:main results} the main results of the paper. The delicate and rather technical proof of the Schur-like property, Theorem \ref{thrm:weak-implies-strong convergence}, is provided in Section \ref{sec:weak-implies-strong}. It uses a kind of geometric argument, inspired by the work of Szarek (see \cite{KPS, Lasota-Szarek:2006}), that enables a tightness argument essentially. Note that our approach yields a new, independent and self-conatined proof of the $\mathrm{U}_b(S)$-weak sequential completeness of $\CM(S)$ (cf. \cite{Pachl:1979}, or \cite{Pachl:2013}, Theorem 5.45) as corollary. Section \ref{sec:weak sequ completeness} shows that the Schur-like property also implies -- for Polish spaces -- the well-known fact of $\sigma(\CM(S),C_b(S))$-weakly sequentially completeness of $\CM(S)$. It uses a type of argument that is of independent interest.

\section{Preliminaries}
\label{sec:preliminaries}

We start with some preliminary results on Lipschitz functions on a metric space $(S,d)$. We denote the vector space of all real-valued Lipschitz functions by $\Lip(S)$. The Lipschitz constant of $f\in\Lip(S)$ is 
\[
|f|_L := \sup\left\{ \frac{|f(x)-f(y)|}{d(x,y)}\,:\, x,y\in S,\ x\neq y\right\}.
\]
$\BL(S)$ is the subspace of bounded functions in $\Lip(S)$. It is a Banach space when equipped with the bounded Lipschitz or Dudley norm
\[
\|f\|_\BL := \|f\|_\infty + |f|_L.
\] 
The norm $\|f\|_\FM:=\max(\|f\|_\infty,|f|_L))$ is equivalent. $\BL(S)$ is partially ordered by pointwise ordering.

The space $\CM(S)$ embeds into $\BL(S)^*$ by means of integration: $\mu\mapsto I_\mu$, where
\[
I_\mu(f) = \pair{\mu}{f} := \int_S f\,d\mu.
\]
The norms on $\BL(S)^*$ dual to either $\|\cdot\|_\BL$ or $\|\cdot\|_\FM$ introduce equivalent norms on $\CM(S)$ through the map $\mu\mapsto I_\mu$. These are called the bounded Lipschitz norm, or Dudley norm, and Fortet-Mourier norm on $\CM(S)$, respectively. $\CM(S)$ equipped with the norm topology induced by either of these norms is denoted by $\CM(S)_\BL$. It is not complete generally. We write $\|\cdot\|_\TV$ for the total variation norm on $\CM(S)$:
\[
\|\mu\|_\TV = |\mu|(S) = \mu^+(S) + \mu^-(S),
\]
where $\mu=\mu^+-\mu^-$ is the Jordan decomposition of $\mu$. $\CM^+(S)$ is the convex cone of positive measures in $\CM(S)$. One has
\begin{equation}
\|\mu\|_\TV = \|\mu\|_\BL^* = \|\mu\|_\FM^*\qquad \mbox{for all}\ \mu\in\CM^+(S).
\end{equation}
In general, for $\mu\in\CM(S)$, $\|\mu\|_\BL^*\leq \|\mu\|_\FM^*\leq \|\mu\|_\TV$.

A finite signed Borel measure $\mu$ is {\it tight} if for every $\eps>0$ there exists a compact set $K_\eps\subset S$ such that $|\mu|(S\setminus K_\eps)<\eps$. A family $M\subset\CM(S)$ is {\it tight} or {\it uniformly tight} if for every $\eps>0$ there exists a compact set $K_\eps\subset S$ such that $|\mu|(S\setminus K_\eps)<\eps$ for all $\mu\in M$. According to Prokhorov's Theorem (see \cite{Bogachev-II:2007}, Theorem 8.6.2), if $(S,d)$ is a complete separable metric space, a set of Borel probability measures $M\subset\CP(S)$ is tight if and only if it is precompact in $\CP(S)_\BL$. Completeness of $S$ is an essential condition for this theorem to hold.
\vskip 0.2cm

In a metric space $(S,d)$, if $A\subset S$ is nonempty, we write
\[
A^\eps := \{x\in S: d(x,A)\leq \eps\}
\]
for the closed $\eps$-neighbourhood of $A$.

\section{Main results}
\label{sec:main results}

A fundamental result on the weak topology on signed measures induced by this pairing is the following fundamental result that provides a `{\it weak-implies-strong-convergence}' property for pairing with $\BL(S)$ on which we build our main results:

\begin{theorem}[Schur-like property]
\label{thrm:weak-implies-strong convergence}
Let $(S,d)$ be a complete separable metric space. Let $(\mu_n)\subset \CM(S)$ be such that $\sup_n\|\mu_n\|_\TV<\infty$. If for every 
$f\in\BL(S)$ the sequence $\pair{\mu_n}{f}$ converges, then there exists $\mu\in\CM(S)$ such that $\|\mu_n-\mu\|_\BL^*\to 0$ as $n\to\infty$.
\end{theorem}

A self-contained, delicate proof of this result is deferred to Section \ref{sec:weak-implies-strong}. The condition that the sequence of measures must be bounded in total variation norm cannot be omitted as the following counterexample indicates.

\begin{counterexample}\label{counter:Schur prop}
Let $S=[0,1]$ with the Euclidean metric. Let $d\mu_n:= n\sin(2\pi n x)\,dx$, where $dx$ is Lebesgue measure on $S$. Then $\|\mu_n\|_\TV$ is unbounded. Let $g\in\BL(S)$ with $|g|_L\leq 1$. According to Rademacher's Theorem, $g$ is differentiable Lebesgue almost everywhere. Since $|g|_L\leq 1$, there exists $f\in L^\infty([0,1])$ such that for all $0\leq a<b\leq 1$,
\[
\int_a^b f(x)\,dx = g(b)-g(a).
\]
This yields
\[
\pair{\mu_n}{g} = \frac{1}{2\pi} \int_0^1 \cos(2\pi nx) f(x)\,dx.
\]
Since $f\in L^2([0,1])$, it follows from Bessel's Inequality that
\[
\lim_{n\to\infty}\ \int_0^1 \cos(2\pi nx) f(x)\,dx =0.
\]
So $\pair{\mu_n}{g}\to 0$ for all $g\in\BL(S)$. Now, let $g_n\in\BL(S)$ be the piecewise linear function that satisfies $g_n(0)=0=g_n(1)$,
\[
g_n\bigl(\mbox{$\frac{1+4i}{4n}$}\bigr) = \mbox{$\frac{1}{4n}$},\qquad 
g_n\bigl(\mbox{$\frac{3+4i}{4n}$}\bigr) = -\mbox{$\frac{1}{4n}$},\qquad
\mbox{for}\ i\in\NN,\ 0\leq i\leq n-1.
\]
Then $|g|_L=1$ and $\|g_n\|_\infty=\frac{1}{4n}$. An easy calculation shows that $\pair{\mu_n}{g_n}=\frac{1}{\pi^2}$ for all $n\in\NN$. Therefore $\|\mu_n\|_\BL^*$ cannot converge to zero as $n\to\infty$.
\end{counterexample}

Theorem \ref{thrm:weak-implies-strong convergence} has the following corollary. Here we denote by $\mathrm{U}_b(S)$ the Banach space of uniformly continuous bounded functions on $S$, equipped with the $\|\cdot\|_\infty$-norm. This result was originally obtained by Pachl \cite{Pachl:1979}, see also \cite{Pachl:2013}, Theorem 5.45.
\begin{corollary}
$\CM(S)$ is $\mathrm{U}_b(S)$-weakly sequentially complete.
\end{corollary}
\begin{proof}
Let $(\mu_n)\subset\CM(S)$ be such that $\pair{\mu_n}{f}$ is Cauchy for every $f\in\mathrm{U}_b(S)$. Then it follows from the Uniform Boundedness Principle that the sequence $(\mu_n)$ is bounded in $\mathrm{U}_b(S)^*$. Consequently, $\sup_n\|\mu_n\|_\TV = M <\infty$. Theorem \ref{thrm:weak-implies-strong convergence} implies that there exists $\mu\in\CM(S)$ such that $\pair{\mu_n}{f}\to\pair{\mu}{f}$ for every $f\in \BL(S)$. Since $\BL(S)$ is dense in $\mathrm{U}_b(S)$ (\cite{Dudley:1966}, Lemma 8) and $\|\mu_n\|_\TV\leq M$ for all $n$, the convergence result holds for every $f\in\mathrm{U}_b(S)$.
\end{proof}

\begin{remark}
Theorem \ref{thrm:weak-implies-strong convergence} is related to results on asymptotic proximity of sequences of distributions, e.g. see \cite{Davydov-Rotar:2009}, Theorem 4. In that setting $\mu_n=P_n-Q_n$, where $P_n$ and $Q_n$ are probability measures. These are asymptotically proximate (for the $\|\cdot\|_\BL^*$-norm; other norms are considered as well) if $\|P_n-Q_n\|_\BL^*\to 0$. So one knows in advance that $\pair{\mu_n}{f}\to 0$. That is, the limit measure $\mu$ exists: $\mu=0$. Combining such a result with the $\mathrm{U}_b(S)$-weak sequential completeness of $\CM(S)$ implies Theorem \ref{thrm:weak-implies-strong convergence}. We present, in Section \ref{sec:weak-implies-strong}, an independent proof using completely different methods, that results in both the completeness result and a particular case of the mentioned asymptotic proximity result. The limit measure is there obtained through a delicate tightness argument, essentially.
\end{remark}

The statement of the particular case in which all measures are positive seems novel too: 
\begin{theorem}\label{thrm:weak-implies-strong positive}
Let $(S,d)$ be a complete separable metric space. Let $(\mu_n)\subset\CM^+(S)$ be such that for every $f\in \BL(S)$, $\pair{\mu_n}{f}$ converges. Then $\pair{\mu_n}{f}$ converges for every $f\in C_b(S)$. In particular, there exists $\mu\in\CM^+(S)$ such that $\|\mu_n-\mu\|_\BL^*\to 0$.  
\end{theorem}
Its proof is simpler compared to that of Theorem \ref{thrm:weak-implies-strong convergence}. In Section \ref{sec:weak-implies-strong} we shall present a self-contained proof of this result as well, based on a `set-geometric' argument that is (essentially) also used to prove Theorem \ref{thrm:weak-implies-strong convergence}.

As it turned out, the proof for signed measures cannot be reduced straightforwardly to the result for positive measures. This is mainly caused by the complication, that for a sequence $(\mu_n)$ of signed measures such that $\pair{\mu_n}{f}$ that is convergent for every $f\in\BL(S)$, it need not hold that $\pair{\mu^+_n}{f}$ and $\pair{\mu^-_n}{f}$ converge for every $f\in\BL(S)$. Take for example on $S=\RR$ with the usual Euclidean metric $\mu_n:=\delta_n-\delta_{n+\frac{1}{n}}$. Then $\pair{\mu_n}{f}\to 0$ for every $f\in\BL(\RR)$. However, $\mu^+_n=\delta_n$ and $\mu^-_n=\delta_{n+\frac{1}{n}}$, so $\pair{\mu^\pm_n}{f}$ will not converge for every $f\in\BL(\RR)$. Thus, an immediate reduction to positive measures is not possible.

The pairing of measures with bounded Lipschitz functions is precisely what is important for the study of Markov operators and semigroups that have particular equicontinuity properties, as we shall discuss next.

\subsection{Equicontinuous families of Markov operators}
\label{sec:equicontinous}

A {\it Markov operator} on (measures on) $S$ is a map $P:\CM^+(S)\to\CM^+(S)$ such that:
\begin{enumerate}
\item $P(\mu+\nu)= P\mu + P\nu$ and $P(r\mu) = rP\mu$ for all $\mu,\nu\in\CM^+(S)$ and $r\geq 0$,
\item $(P\mu)(S) = \mu(S)$ for all $\mu\in\CM^+(S)$.
\end{enumerate}
In particular, a Markov operator leaves invariant the convex set $\CP(S)$ of probability measures in $\CM^+(S)$. Let $\mathrm{BM}(S)$ be the vector space of bounded Borel measurable real-valued functions on $S$. A Markov operator is called {\it regular} if there exists a linear map $U:\mathrm{BM}(S)\to\mathrm{BM}(S)$, the {\it dual operator}, such that
\[
\pair{P\mu}{f} = \pair{\mu}{Uf}\qquad \mbox{for all}\ \mu\in\CM^+(S),\ f\in\mathrm{BM}(S).
\]
A regular Markov operator $P$ is {\it Feller} if its dual operator maps $C_b(S)$ into itself. Equivalently, $P$ is continuous for the $\|\cdot\|_\BL^*$-norm topology (cf. e.g. \cite{Hille-Worm:2009a} Lemma 3.3 and \cite{Worm-thesis:2010} Lemma 3.3.2). 

Regular Markov operators on measures appear naturally e.g. in the theory of Iterated Function Systems \cite{Barnsley_ea:1988, Lasota-Yorke:1994} and the study of deterministic flows by their lift to measures \cite{Piccoli-Tosin:2011, Evers_ea:2015}. Dual Markov operators on $C_b(S)$ (or a suitable linear subspace) are encountered naturally in the study of stochastic differential equations \cite{DaPrato-Zabczyk:1992, KPS}. Which specific viewpoint in this duality is used, is often determined by technical considerations and the mathematical problems that are considered. 

Markov operators and semigroups with equicontinuity properties (called the `{\it e-property}') have convenient properties concerning existence, uniqueness and asymptotic stability of invariant measures, see e.g. \cite{HHS:2016, KPS, SSU:2010, Szarek-Worm:2012, Worm-thesis:2010}. After having defined these properties precisely below, we show by means of Theorem \ref{thrm:weak-implies-strong convergence} that a dual viewpoint exists for equicontinuity too, in Theorem \ref{thrm:equicvalence equicont}. In subsequent work further consequences of this result for the theory and application of equicontinuous families of Markov operators are examined. Some results in this direction were also discussed in parts of \cite{Worm-thesis:2010}, Chapter 7.

Let $T$ be a topological space and $(S',d')$ a metric space. A family of functions $\CE\subset C(T,S')$ is {\it equicontinuous at $t_0\in T$} if for every $\eps>0$ there exists an open neighbourhood $U_\eps$ of $t_0$ such that 
\[
d'(f(t),f(t_0))<\eps\quad \mbox{for all}\ f\in\CE,\ t\in U_\eps.
\]
$\CE$ is {\it equicontinuous} if it is equicontinuous at every point of $T$.

Following Szarek {\it et al.} \cite{KPS,SSU:2010}, a family $(P_\lambda)_{\lambda\in\Lambda}$ of regular Markov operators has the {\it e-property} if for each $f\in\BL(S)$ the family $\{U_\lambda f: \lambda\in\Lambda\}$ is equicontinuous in $C_b(S)$. In particular one may consider the family of iterates of a single Markov operator $P$: $(P^n)_{n\in\NN}$, or Markov semigroups $(P_t)_{t\in\RR^+}$, where each $P_t$ is a regular Markov operator and $P_0=I$, $P_t P_s = P_{t+s}$.   

Our main result on equicontinuous families of Markov operators is

\begin{theorem}\label{thrm:equicvalence equicont}
Let $\{P_\lambda:\lambda\in\Lambda\}$ be a family of regular Markov operators on a complete separable metric space $(S,d)$. Let $U_\lambda$ be the dual Markov operator of $P_\lambda$. The following statements are equivalent:
\begin{enumerate}
\item $\{U_\lambda f: \lambda\in \Lambda\}$ is equicontinuous in $C_b(S)$ for every $f\in\BL(S)$.
\item $\{P_\lambda:\lambda\in\Lambda\}$ is equicontinuous in $C(\CM^+(S)_\BL,\CM^+(S)_\BL)$,
\item $\{P_\lambda:\lambda\in\Lambda\}$ is equicontinuous in $C(\CP(S)_{\mathrm{weak}},\CP(S)_\BL)$
\end{enumerate}
\end{theorem}

\begin{proof}
{\it (i) $\Rightarrow$ (ii).}\ Assume on the contrary that $\{P_\lambda:\lambda\in\lambda\}$ is not an equicontinuous family of maps. Then there exists a point $\mu_0\in \CM^+(S)$ at which this family is not equicontinuous. Hence there exists $\eps_0>0$ such that for every $k\in\NN$ there are $\lambda_k\in\Lambda$ and $\mu_k\in\CM^+(S)$ such that
\begin{equation}\label{eq:contradictive assumption}
\|\mu_k-\mu_0\|_\BL^*<\mbox{$\frac{1}{k}$} \quad\mbox{and}\quad 
\|P_{\lambda_k}\mu_k-P_{\lambda_k}\mu_0\|_\BL^*\geq\eps_0
\qquad\mbox{for all}\ k\in\NN.
\end{equation}
Because the measures $\mu_k$ are positive and the $\|\cdot\|_\BL^*$-norm metrizes the $C_b(S)$-weak topology on $\CM^+(S)$ (cf. \cite{Dudley:1966}, Theorem 18), $\pair{\mu_k}{f}\to\pair{\mu_0}{f}$ for every $f\in C_b(S)$. According to \cite{Dudley:1966}, Theorem 7, this convergence is uniform on any equicontinuous and uniformly bounded subset $\CE$ of $C_b(S)$. By assumption, $\CM_f:=\{U_{\lambda_k}f:k\in\NN\}$ is such a family for every $f\in\BL(S)$. Therefore
\begin{equation}\label{eq:conv to zero}
|\pair{P_{\lambda_k}\mu_k-P_{\lambda_k}\mu_0}{f}| = |\pair{\mu_k-\mu_0}{U_{\lambda_k}f}|\to 0
\end{equation}
as $k\to\infty$ for every $f\in\BL(S)$. Since for positive measures $\mu$ one has $\|\mu\|_\TV= \|\mu\|_\BL^*$, one obtains
\[
\bigl| \|\mu_k\|_\TV - \|\mu_0\|_\TV\bigr| \leq \|\mu_k-\mu_0\|_\BL^*\to 0.
\]
So $m_0:=\sup_{k\geq 1}\|\mu_k\|_\TV<\infty$. Moreover, 
\[
\|P_{\lambda_k}\mu_k-P_{\lambda_k}\mu_0\|_\TV \leq \|P_{\lambda_k}\mu_k\|_\TV + \|P_{\lambda_k}\mu_0\|_\TV \leq \|\mu_k\|_\TV + \|\mu_0\|_\TV \leq m_0+\|\mu_0\|_\TV.
\]
Theorem \ref{thrm:weak-implies-strong convergence} and \eqref{eq:conv to zero} yields that $\|P_{\lambda_k}\mu_k-P_{\lambda_k}\mu_0\|_\BL^*\to 0$ as $k\to \infty$. This contradicts the second property in \eqref{eq:contradictive assumption}.\\
{\it (ii) $\Rightarrow$ (iii).}\ Follows immediately by restriction of the Markov operators $P_\Lambda$ to $\CP(S)$.\\
{\it (iii) $\Rightarrow$ (i).}\ Let $f\in\BL(S)$ and $x_0\in S$. Let $\eps>0$. Since $\{P_\lambda:\lambda\in\Lambda\}$ is equicontinuous at $\delta_{x_0}$ there exists an open neighbourhood $V$ of $\delta_{x_0}$ in $\CP(S)_{\mathrm{weak}}$ such that 
\[
\|P_\lambda\delta_{x_0} - P_\lambda\mu\|_\BL^* <\eps/(1+\|f\|_\BL)\quad\mbox{for all}\  \lambda\in\Lambda\ \mbox{and}\ \mu\in U_0.
\]
Since the map $x\mapsto\delta_x:S\to\CP(S)_{\mathrm{weak}}$ is continuous, there exists an open neighbourhood $V_0$ of $x_0$ in $S$ such that $\delta_x\in V$ for all $x\in V_0$. Then
\[
|U_\lambda f(x)-U_\lambda f(x_0)| = |\pair{P_\lambda \delta_x - P_\lambda\delta_{x_0}}{f}|\leq \frac{\eps}{1+\|f\|_\BL}\cdot\|f\|_\BL<\eps
\]
for all $x\in V_0$ and $\lambda\in\Lambda$.
\end{proof}

A particular class of examples of Markov operators and semigroups is furnished by the lift of a map or semigroup $(\phi_t)_{t\geq 0}$ of measurable maps $\phi_t:S\to S$ to measures on $S$ by means of push-forward:
\[
P^\phi_t\mu (E) := \mu\bigl( \phi_t^{-1}(E)\bigr)
\]
for every Borel set $E$ of $S$ and $\mu\in\CM^+(S)$. A consequence of Theorem \ref{thrm:equicvalence equicont} is:

\begin{proposition}\label{prop:equicont lift}
Let $(S,d)$ be a complete separable metric space and let $(\phi_t)_{t\geq 0}$ be a semigroup of Borel measurable transformations of $S$. Then $P^\phi_t$ is a regular Markov operator for each $t\geq 0$. Moreover, $(P^\phi_t)_{t\geq 0}$ is equicontinuous in $C(\CM^+(S)_\BL,\CM^+(S)_\BL)$ if and only if $(\phi_t)_{t\geq 0}$ is equicontinuous in $C(S,S)$.
\end{proposition}
\begin{proof} 
The regularity of $P^\phi_t$ is immediate, as $U^\phi_t f= f\circ\phi_t$.\\
`$\Rightarrow$':\ Let $x_0\in S$ and $\eps>0$. Define $h(x):=2x/(2+x)$ and put $\eps':=h(\eps)$. By equicontinuity of $(P^\phi_t)_{t\geq 0}$ at $\delta_{x_0}$, there exists and open neighbourhood $U$ of $\delta_{x_0}$ in $\CM^+(S)_\BL$ such that
\[
\| P^\phi_t \mu - P^\phi_t \delta_{x_0}\|_\BL^* <\eps'
\]
for all $t\geq 0$ and $\mu \in U$. Because the map $\delta:x\mapsto \delta_x:S\to\CM^+(S)_\BL$ is continuous, $U_0:=\delta^{-1}(U)$ is open in $S$. It contains $x_0$. Moreover,
\[
\| P^\phi_t \delta_x - P^\phi_t \delta_{x_0}\|_\BL^* = \|\delta_{\phi_t(x)} - \delta_{\phi_t(x_0)}\|_\BL^* = h\bigl(d(\phi_t(x),\phi_t(x_0)\bigr) <\eps'
\]
for all $x\in U_0$ and $t\geq 0$ (see \cite{Hille-Worm:2009} Lemma 3.5). Because $h$ is monotone increasing, 
\[
d\bigl(\phi_t(x),\phi_t(x_0)\bigr) < \eps\qquad \mbox{for all}\ x\in U_0,\ t\geq 0.
\]
`$\Leftarrow$':\ This part involves Theorem \ref{thrm:equicvalence equicont}. Let $f\in \BL(S)$. Let $U_t$ be the dual operator of $P_t$. Then for all $x,x_0\in S$,
\[
|U_t f(x) - U_t f(x_0)| = |f(\phi_t(x))-f(\phi_t(x_0))| \leq |f|_L d(\phi_t(x),\phi_t(x_0)),
\]
from which the equicontinuity of $\{U_tf: t\geq 0\}$ follows. The result is obtained by applying Theorem \ref{thrm:equicvalence equicont}.
\end{proof}

\subsection{Coincidence of weak and norm topologies}
\label{sec:coincidence topologies}

A further consequence of Theorem \ref{thrm:weak-implies-strong convergence} is

\begin{theorem}\label{thrm:coincidence topologies}
Let $(S,d)$ be a complete separable metric space and let $M\subset\CM(S)$ be such that $m:=\sup_{\mu\in M}\|\mu\|_\TV<\infty$. If the restriction of the $\sigma(\CM(S),\BL(S))$-weak topology to $M$ is first countable, then this topology coincides with the restriction of the $\|\cdot\|_\BL^*$-norm topology to $M$. 
\end{theorem}

\begin{proof}
We have to show that for any $\|\cdot\|_\BL^*$-norm closed set $C$, $C\cap M$ is closed in the restriction of the $\sigma(\CM(S),\BL(S))$-weak topology to $M$. Since the latter is first countable, $C\cap M$ is relatively $\sigma(\CM(S),\BL(S))$-weak closed if and only if for every $\sigma(\CM(S),\BL(S))$-weakly converging {\it sequence} $\mu_n\to\mu$ in $\CM(S)$ with $\mu_n\in C$, one has $\mu\in C$ (cf. \cite{Kelley:1955} Theorem 2.8, p. 72). Let $(\mu_n)$ be such a sequence. Because $\sup_{\mu\in M}\|\mu\|_\TV<\infty$ by assumption, Theorem \ref{thrm:weak-implies-strong convergence} implies that there exists $\mu'\in\CM(S)$ such that $\|\mu_n-\mu'\|_\BL^*\to 0$. Since $C$ is relatively $\|\cdot\|_\BL^*$-norm closed in $M$, $\mu'\in C$. Moreover, $\pair{\mu}{f}=\pair{\mu'}{f}$ for every $f\in\BL(S)$, so $\mu=\mu'\in C$.
\end{proof}

The following technical result provides a tractable condition that ensures first countability of the relative weak topology on the set $M$, as we shall show after having proven the result. We need to introduce some notation. For $\lambda>0$ and $C\subset S$ closed and nonempty, define
\[
h_{\lambda,C}(x) := \bigl[ 1- \mbox{$\frac{1}{\lambda}$} d(x,C)\bigr]^+.
\]
Then $h_{\lambda, C}\in \BL(S)$, $|h_{\lambda, C}|_L=\frac{1}{\lambda}$, $0\leq h_{\lambda, C}\leq 1$ and $h_{\lambda, C}\downarrow \ind_C$ pointwise as $\lambda\downarrow 0$. Moreover $h_{\lambda, C}=0$ on $S\setminus C^\lambda$.
We can now state the result.

\begin{lemma}\label{lem:tractable cond}
Let $M\subset \CM(S)$ be such that $m:=\sup_{\mu\in M}\|\mu\|_\TV<\infty$. If for every $\mu\in M$ and every $\eps>0$ there exist $K_1,\dots,K_n\subset S$ compact such that for $K=\bigcup_{i=1}^n K_i$:
\begin{enumerate}
\item $|\mu|(S\setminus K)<\eps$,
\item There exists $0<\lambda_0\leq \eps$ such that for all $0<\lambda\leq\lambda_0$ there exists $\delta_1,\dots, \delta_n>0$ such that the following statement holds:
\begin{align*}
&\mbox{If}\ \nu\in M\ \mbox{satisfies}\ |\pair{\mu-\nu}{h_{\lambda,K_i}}|<\delta_i\ \mbox{for all $i=1,\dots,n$},\\
&\mbox{then}\ |\nu|(S\setminus K^\lambda)<\eps.
\end{align*}
\end{enumerate}
Then the relative $\sigma(\CM(S),\BL(S))$-weak topology on $M$ is first countable.
\end{lemma}

\begin{proof}
We first define a countable family $\CF$ of functions in $\bar{B}:=\{g\in\BL(S):\|g\|_\infty\leq 1\}$ that is dense in $\bar{B}$ for the compact-open topology, i.e. the topology of uniform convergence on compact subsets of $S$. Let $D$ be a countable dense subset of $S$. The family of {\it finite} subsets of $D$ is countable. Let $I_\QQ:=\QQ\cap[0,1]$. For a finite subset $F\subset D$, $\lambda\in I_\QQ\setminus \{0\}$ and function $a:F\to I_\QQ$ define
\[
f_{F,a}^\lambda(x) := \bigvee_{y\in F} \bigl[ a(y) (1 - \mbox{$\frac{1}{\lambda}$}d(x,y))^+\bigr].
\]
Here $\vee$ denotes the maximum, as before. Then $f_{F,a}^\lambda\in\BL(S)$, $|f_{F,a}^\lambda|_L\leq \max_{y\in F}\frac{a(y)}{\lambda} \leq \frac{1}{\lambda}$. Moreover, $f^\lambda_{F,a}$ vanishes outside $F^\lambda= \bigcup_{y\in F} B(y,\lambda)$.
For a finite subset $F\subset D$ the family $\CF_F$ of all such functions $f^\lambda_{F,a}$ with $a$ and $\lambda$ as indicated is countable. So the union $\CF^+$ of all sets $\CF_F$ over all finite $F\subset D$ is countable too. It is quickly verified that on any compact subset $K$ of $S$ any positive $h\in\bar{B}$ can be uniformly approximated by $f\in\CF^+$. Consequently, $\CF = \CF^+-\CF^+\subset \BL(S)$ is countable and any $h\in \bar{B}$ can be approximated uniformly on compact sets by means of $f\in \CF$.

Now let $\mu\in M$ and consider the open neighbourhood 
\[
U_\mu(h,r) := \bigl\{\nu\in M: |\pair{\mu-\nu}{h}|<r\bigr\},
\]
with $r>0$ and $h\in\BL(S)$. Without loss of generality we can assume that $\|h\|_\BL=1$. We shall prove that there exist $f_0, \dots, f_n\in \CF$ and $q_0, \dots, q_n>0$ in $\QQ$ such that
\begin{equation}\label{eq:subbase}
\bigcap_{i=0}^n \;\bigl\{\nu\in M: \bigl|\pair{\mu-\nu}{h_i}\bigr|<q_i\bigr\} \subset U_\mu(h,r).
\end{equation}
Then the relative weak topology on $M$ is first countable.

Let $\eps\in\QQ$ such that $0<\eps\leq\frac{1}{6}r$ and let $K_i, K\subset S$ be compact and $0<\lambda_0\leq\eps$ as in the conditions of the lemma. There exists $f_0\in\CF$ such that $\sup_{x\in K}|h(x)-f_0(x)|\leq \frac{1}{4m}\eps$. Then for any $0<\lambda\leq\lambda_0$, $x\in K^\lambda$ and $x_0\in K$,
\begin{align*}
|h(x)-f_0(x)| &\leq |h(x)-h(x_0)|\ +\ |h(x_0)-f_0(x_0)|\ + \ |f_0(x_0)-f_0(x)| \\
&\leq (1+|f_0|_L)d(x,x_0) + \mbox{$\frac{1}{4m}$}\eps.
\end{align*}
Hence
\[
\sup_{x\in K^\lambda} |h(x)-f_0(x)| \leq (1+|f_0|_L)\lambda + \mbox{$\frac{1}{4m}$}\eps.
\]
Let $0<\lambda'_0\leq \lambda_0$ be such that $(1+|f_0|_L)\lambda'_0\leq\frac{1}{4m}\eps$. Now one has, using property ({\it i}\,),
\begin{align}
|\pair{\mu-\nu}{h}| & \leq |\pair{\mu-\nu}{h-f_0}| + |\pair{\mu-\nu}{f_0}|\nonumber\\
& \leq \int_{K^\lambda} |h-f_0|\,d|\mu-\nu|\ +\ 2|\mu|(S\setminus K^\lambda) + 2|\nu|(S\setminus K^\lambda) + |\pair{\mu-\nu}{f_0}|\nonumber\\
& \leq \mbox{$\frac{1}{2m}$}\eps\cdot 2m + 2\eps + 2|\nu|(S\setminus K^\lambda) + |\pair{\mu-\nu}{f_0}| \label{eq:estimate pairing h}
\end{align}
for all $0<\lambda\leq \lambda'_0$. Fix $\lambda\in\QQ$ with $0<\lambda\leq \lambda'_0$ and let $\delta_1,\dots,\delta_n$ be as in property ({\it ii}\,). 

The Hausdorff semidistance on closed and bounded subsets of $S$ is given by
\[
\delta(C,C') := \sup_{x\in C} d(x,C').
\]
The Hausdorff distance is defined by 
\[
d_H(C,C') := \max\bigl(\delta(C,C'),\delta(C',C)\bigr).
\]
The collection of finite subsets of $D$ form a separable dense subset of the set of compact subsets of $S$, $\CK(S)$, for $d_H$. If $F\subset D$ is finite and $K'\in \CK(S)$, then by the Birkhoff Inequalities
\begin{align*}
|h_{\lambda,K'} - h_{\lambda, F}| & = \bigl| \bigl[1-\mbox{$\frac{1}{\lambda}$}d(x,K')\bigr]^+ - \bigl[1-\mbox{$\frac{1}{\lambda}$}d(x,F)\bigr]^+|\\
& \leq \bigl| \bigl[1-\mbox{$\frac{1}{\lambda}$}d(x,K')\bigr] - \bigl[1-\mbox{$\frac{1}{\lambda}$}d(x,F)\bigr] \bigr|\\
& = \mbox{$\frac{1}{\lambda}$} |d(x,K')- d(x,F)| \leq \mbox{$\frac{1}{\lambda}$} \cdot d_H(K',F).
\end{align*}
Let $F_i\subset D$ be finite such that $d_H(K_i,F_i)\leq \frac{1}{4m}\lambda\delta_i$. Then $h_{\lambda,F_i}=f^\lambda_{F_i,\ind}\in\CF$. Put $f_i:=h_{\lambda,F_i}$. Let $q_i\in\QQ$ be such that $0<q_i<\frac{1}{2}\delta_i$. If $\nu\in M$ is such that $|\pair{\mu-\nu}{f_i}|<q_i$ for $i=1,\dots,n$, then
\[
|\pair{\mu-\nu}{h_{\lambda, K_i}}| \leq \|h_{\lambda,K_i}-h_{\lambda,F_i}\|_\infty\cdot \|\mu-\nu\|_\TV + |\pair{\mu-\nu}{f_i}| < \mbox{$\frac{1}{2}$}\delta_i + \mbox{$\frac{1}{2}$}\delta_i = \delta_i
\]
According to condition ({\it ii}\,) one has $|\nu|(S\setminus K^\lambda)<\eps$. Put $q_0=\eps$. Inequality \eqref{eq:estimate pairing h} then yields
\eqref{eq:subbase}, as desired.
\end{proof}

Because conditions ({\it i}\,) and ({\it ii}\,) in Lemma \ref{lem:tractable cond} are immediately satisfied when $M$ is uniformly tight, we obtain

\begin{corollary}\label{clry:coincidence tight}
Let $(S,d)$ be a complete separable metric space and let $M\subset\CM(S)$ such that $\sup_{\mu\in M} \|\mu\|_\TV<\infty$ and $M$ is uniformly tight. Then the $\sigma(\CM(S),\BL(S))$-weak topology coincides with the $\|\cdot\|_\BL^*$-norm topology on $M$.
\end{corollary}

\begin{remark} 
Gwiazda {\it et al.} \cite{Gwiazda_ea:2010} state at p. 2708 that the topology of narrow convergence in $\CM(S)$, i.e. that of convergence of sequences of signed measures paired with $f\in C_b(S)$, is metrizable on tight subsets that are uniformly bounded in total variation norm. In fact it can be metrized by the norm $\|\cdot\|_\BL^*$.
\end{remark}

A second case, more involved, in which the conditions of Lemma \ref{lem:tractable cond} are satisfied, is: 

\begin{proposition}\label{clry:coincidence spheres}
Let $(S,d)$ be a complete separable metric space and let 
\[
M:=\{\mu\in \CM(S): \|\mu\|_\TV=\rho\},\quad (\rho>0).
\]
Then condition ({\it i}) and ({\it ii}) of Lemma \ref{lem:tractable cond} hold. In particular, the relative $\sigma(\CM(S),\BL(S))$-weak topology  and relative $\|\cdot\|_\BL^*$-norm topology on $M$ coincide. 
\end{proposition}

\begin{proof}
Take $\eps>0$, $\mu\in M$ and let $\mu^+$ and $\mu^-$ be the positive and negative part of $\mu$, i.e. $\mu=\mu^+-\mu^-$. Since $\mu^\pm$ are disjoint and tight, by Ulam's Lemma, there exist compact sets $K_\pm\subset S$ such that $K_+\cap K_-=\emptyset$, $\mu^\pm(K_\mp)=0$ and
\begin{equation}\label{eq:inital tightness estimate}
\mu^+(S)-\mu^+(K_+)<\eps/8\quad \mbox{and}\quad \mu^-(S)-\mu^-(K_-)<\eps/8.
\end{equation}
In particular,
\[
|\mu|(S\setminus (K_+\cup K_-)) \leq \mu^+(S\setminus K_+) +\mu^-(S\setminus K_-) < \mbox{$\frac{1}{8}$}\eps + \mbox{$\frac{1}{8}$}\eps <\eps,
\]
so condition ({\it i}\,) of Lemma \ref{lem:tractable cond} is satisfied for $K=K_+\cup K_-$.

Because $K_+$ and $K_-$ are compact, there exists $\lambda_0>0$ such that $K_+^{\lambda_0}\cap K_-^{\lambda_0}=\emptyset$. Then $K_+^\lambda\cap K_-^\lambda=\emptyset$ for all $0<\lambda\leq\lambda_0$.
Without loss of generality we can assume that $\lambda_0\leq\eps$. Fix $0<\lambda\leq \lambda_0$.

Let us assume for the moment that $\delta_\pm>0$ have been selected. At the end we will then see how to choose these, such that condition ({\it ii}\,) will be satisfied.
If $\nu\in M$ satisfies 
\begin{equation}\label{condition nu}
|\pair{\mu-\nu}{h_{\lambda,K_+}}|<\delta_+\quad \mbox{and}\quad |\pair{\mu-\nu}{h_{\lambda,K_-}}|<\delta_-,
\end{equation}
then
\[
\pair{\mu-\nu^+}{h_{\lambda,K_+}} \leq \pair{\mu-\nu^++\nu^-}{h_{\lambda,K_+}}\leq |\pair{\mu-\nu}{h_{\lambda,K_+}}|<\delta_+.
\]
Consequently, since $\ind_{K_+}\leq h_{\lambda,K_+} \leq\ind_{K^\lambda_+}$,
\[
\mu^+(K_+) - \mu^-(K^\lambda_+)-\nu^+(K^\lambda_+) \leq \pair{\mu-\nu^+}{h_{\lambda,K_+}}<\delta_+.
\]
We obtain
\begin{align*}
\nu^+(K^\lambda_+) &> \mu^+(K_+) - \mu^-(K^\lambda_+) - \delta_+
  \ \geq\ \mu^+(K_+) - \mu^-(S\setminus K_-) -\delta_+\\
	& > \mu^+(K_+) - \mbox{$\frac{1}{8}$}\eps - \delta_+.
\end{align*}
In a similar way,
\[
\pair{-\mu-\nu^-}{h_{\lambda,K_-}}\leq\pair{\nu-\mu}{h_{\lambda,K_-}}<\delta_-,
\]
whence
\[
\nu^-(K^\lambda_-) > \mu^-(K_-) - \mbox{$\frac{1}{8}$}\eps - \delta_-.
\]
Therefore, using \eqref{eq:inital tightness estimate},
\begin{align*}
\nu^+(K^\lambda_+)+\nu^-(K^\lambda_-) &> \mu^+(K_+) + \mu^-(K_-) - \mbox{$\frac{1}{4}$}\eps - (\delta_++\delta_-)\\
&> \mu^+(S) + \mu^-(S) - \mbox{$\frac{1}{2}$}\eps -(\delta_++\delta_-) = \rho - (\delta_++\delta_-+ \mbox{$\frac{1}{2}$}\eps).
\end{align*}
Note that in this last step the assumption that $M$ is a total variation sphere is used in an essential manner. The last inequality implies that
\[
|\nu|(S\setminus K^\lambda) = |\nu|(S) - |\nu|(K^\lambda_+) - |\nu|(K^\lambda_-) \leq \rho - \nu^+(K^\lambda_+) - \nu^-(K^\lambda_-) < \delta_++\delta_-+\mbox{$\frac{1}{2}$}\eps.
\]
Thus, if we take $K_1=K_+$, $K_2=K_-$, $\delta_+=\delta_-=\delta_i=\frac{1}{4}\eps$, we see that condition ({\it ii}\,) in Lemma \ref{lem:tractable cond} is satisfied. Theorem \ref{thrm:coincidence topologies} then yields the final statement.
\end{proof}

\begin{remark} 1.) In \cite{Pachl:2013}, Theorem 5.38 and Corollary 5.39 come close to Theorem \ref{thrm:coincidence topologies}. A technical condition seems to prevent deriving our new result on coincidence of topologies from the results in \cite{Pachl:2013}.\\
2.) The result stated in Proposition \ref{clry:coincidence spheres} can be found in \cite{Pachl:2013}, Corollary 5.39. There, a proof of this result is provided using completely different techniques. Concerning coincidence of these topologies on total variation spheres, see some further notes in \cite{Pachl:2013}, indicating e.g. \cite{Granirer-Leinert:1981}.
\end{remark}

In view of Corollary \ref{clry:coincidence tight} and Proposition \ref{clry:coincidence spheres} one might be tempted to conjecture that the weak and norm topologies would coincide on sets of measures with uniformly bounded total variation. This does not hold however, as the following counterexample illustrates.

\begin{counterexample}
Let $(S,d)$ be the natural numbers $\NN$ equipped with the restriction of the Euclidean metric on $\RR$. Now, $\BL(\NN)$ is linearly isomorphic to $\ell^\infty$: the map $f\mapsto (f(n))_{n\in\NN}$ is bijective and continuous. Hence it is a linear isomorphism by Banach's Isomorphism Theorem. Observe that $|f|_L\leq 2\|f\|_\infty$. Since $(\NN,d)$ is uniformly discrete, the norms $\|\cdot\|_\BL^*$ and $\|\cdot\|_\TV$ on $\CM(\NN)$ are equivalent (cf. \cite{Hille-Worm:2009}, proof of Theorem 3.11). So $\CM(\NN)_\BL$ is linearly isomorphic to $\ell^1$ under the map $\mu\mapsto (\mu(\{n\}))_{n\in\NN}$. One has $\|\mu\|_\TV=\|(\mu)\|_{\ell^1}$. Moreover, the duality between $\CM(\NN)$ and $\BL(\NN)$ is precisely the duality between $\ell^1$ and $\ell^\infty$ under the given isomorphisms.
Consider now $M:=\{(\mu)\in\ell^1: \|(\mu)\|_{\ell^1}\leq 1\}$. It represents a set of measures that is uniformly bounded in total variation norm. Let $S:=\{(\mu)\in\ell^1: \|(\mu)\|_{\ell^1}= 1\}$. Then $S$ is a $\|\cdot\|_\TV$-closed subset of $M$. The weak closure of $S$ equals $M$ however (cf. \cite{Conway:1985}, Section V.1, Ex. 10). Therefore, the $\|\cdot\|_\BL^*$ (i.e. $\|\cdot\|_\TV$) and weak topologies cannot coincide on $M$.
\end{counterexample}

\section{Proof of the Schur-like property}
\label{sec:weak-implies-strong}

We provide a self-contained proof of the Schur-like property for spaces of measures, Theorem \ref{thrm:weak-implies-strong convergence}, using a `set-geometric' argument.  See Remark \ref{rem:alternative approaches} below for alternative approaches. 

We first introduce various technical lemmas that enable our set-geometric argument. Then we start with a complete proof of the particular case of positive measures, Theorem \ref{thrm:weak-implies-strong convergence}, as it will aid the reader in getting introduced to the type of argument employed, based on Lemma \ref{lem:non-tight sequence prop}, and the complications that arise when proving the result for general signed measures in the section that follows.

\subsection{Technical lemmas}

The following lemmas are needed in the proof of the fundamental result.

\begin{lemma}\label{lem:sup op BL functions}
Let $A\subset\BL(S)$ be such that $\sup_{f\in A}\|f\|_\BL<\infty$. Then $\sup(A)$ exists in $\BL(S)$ and $|\sup(A)|_L\leq\sup_{f\in A}|f|_L$. In particular, $\|\sup(A)\|_\BL\leq 2\sup_{f\in A}\|f\|_\BL$.
\end{lemma}

\begin{proof}
Put $L:=\sup_{f\in A}|f|_L$ and let $g=\sup(A)$, i.e. $g(x) := \sup\{f(x): f\in A\}$ for every $x\in S$. 
Let $x,y\in S$. We may assume $g(x)\geq g(y)$. Let $\eps>0$. There exists $f\in A$ such that $g(x)<f(x)+\eps$. By definition $g(y)\geq f(y)$. Hence
\[
|g(x)-g(y)| \leq g(x) - f(x)\ +\ f(x) - f(y) < \eps+|f(x)-f(y)| \leq\eps+L\,d(x,y).
\]
Since $\eps$ is arbitrary, we obtain that $|g(x)-g(y)|\leq L d(x,y)$. Thus $g\in\Lip(S)$ and $|g|_L\leq L$. Clearly, $\|g\|_\infty\leq \sup_{f\in A}\|f\|_\infty< \infty$, so $g\in\BL(S)$ and $\|g\|_\BL\leq 2\sup_{f\in A}\|f\|_\BL$.
\end{proof}

The {\it support} of $f\in C(S)$, denoted by $\supp f$, is the closure of the set of points where $f$ is nonzero. Lemma \ref{lem:sup op BL functions} implies the following

\begin{lemma}\label{L1_28.04.16b}
Let $(f_k)\subset \BL(S)$ be such that $\sup_{k\ge 1}\|f_k\|_{\BL}<\infty$. Assume that their supports are pairwise disjoint. Then the series $f(x):=\sum_{k=1}^{\infty} f_k(x)$ converges pointwise and $f\in\BL(S)$. In particular,
\begin{equation}\label{eq:Lipschitz constant estimates}
\|f\|_\infty\leq\sup_{k\geq 1}\|f_k\|_\infty, \qquad |f|_L \leq 2\sup_{k\geq 1}|f_k|_L.
\end{equation}
\end{lemma}
\begin{proof}
Because the sets $\supp f_k$ are pairwise disjoint, $f(x)=f_k(x)$ if $x\in\supp f_k$. So the positive part $f^+$ and negative part $f^-$ of $f$ satisfy $f^\pm = \sum_{k=1}^\infty f^\pm_k$ and it suffices to prove the result for $f\geq 0$. In that case, $f=\sup_{k\geq 1} f_k$, and the first estimate in \eqref{eq:Lipschitz constant estimates} follows immediately. The second follows from Lemma \ref{lem:sup op BL functions}.
\end{proof}

\begin{lemma}\label{lem:non-tight sequence prop}
Let $(S,d)$ be a complete separable metric space. Let $\mu_n\in\CM^{+}(S)$, $n\in\N$. Assume that $\{\mu_n: n\ge 1\}$ is not tight. Then there exists $\eps>0$, an increasing sequence $(n_k)$ of positive integers and a sequence of compact sets $(K_{n_k})$ such that
\[
\mu_{n_k}\bigl(K_{n_k}\bigr)\geq \eps\qquad \mbox{for all}\ k\geq 1
\]
and
\[
\dist (K_{n_k}, K_{n_m}):=\min\{ d(x,y)\,|\, x\in K_{n_k},\ y\in K_{n_m}\} > \eps\qquad \mbox{for all}\ k\neq m.
\]
\end{lemma}

This result was originally stated in \cite{KPS}, Lemma 1, p. 1410, for a sequence $(\mu_n)$ of probability Borel measures with a proof in \cite{Lasota-Szarek:2006} (proof of Theorem 3.1, p. 517-518), but it is also valid for (positive) measures.

In addition to Lemma \ref{lem:non-tight sequence prop} the following observation is made:

\begin{lemma}\label{L1_28.04.16} Let $(\mu_n)\subset \CM^{+}(S)$ be such that $\sup_n \mu_n (S)<\infty$ and let $(E_{n})$ be a sequence of pairwise disjoint Borel measurable subsets of $S$. Then for every $\eps>0$ there exists a strictly increasing subsequence $(n_{i})$ of $\NN$ such that for every $i\ge 1$, 
\begin{equation}
 {\mu}_{n_{i}}\left(\bigcup_{j\neq i} E_{n_{j}}\right)<\eps.
\end{equation}  
\end{lemma}

\begin{proof} Let us first prove that for every $\eta>0$ there exists a strictly increasing subsequence $(m_i)$ such that
 \begin{equation}\label{e2_28.04.16}
 {\mu}_{m_{1}}\left(\bigcup_{i>1} E_{m_{i}}\right)<\eta
 \end{equation}
 and
 \begin{equation}\label{e3_28.04.16}
 {\mu}_{m_{i}}(E_{m_{1}})<\eta\qquad\text{for all $i\ge 2$.}
 \end{equation}
 Fix $\eta>0$. Set $C:= \sup_n \mu_n (S)$ and let $N\ge 1$ be such that $N\eta>C$. 
 Since for every $n\ge 1$  we have $\sum_{m=1}^N{\mu}_n(E_{m})={\mu}_n\left(\bigcup_{m=1}^N E_{m}\right)
 \leq{\mu}_n(S)\le C<N\eta$, there exists $m\in\{1,\ldots, N\}$ such that
 \begin{equation}\label{e1_28.04.16}
 {\mu}_{n} (E_{m})<\eta.
 \end{equation}
 Thus there exists $m_1\in\{1,\ldots, N\}$ and an infinite set $\mathcal S$ such that condition (\ref{e1_28.04.16}) holds for all $n\in \mathcal S$. 
 Let us split $\mathcal S$ into $N$ disjoint infinite subsets $\mathcal S_1,\ldots, \mathcal S_N$.
 
 Since 
 $$
 \bigcup_{n\in \mathcal S_i} E_{n}\cap  \bigcup_{n\in \mathcal S_j} E_{n}=\emptyset\qquad
 \text{for $i, j\in\{1,\ldots, N\}$, $i\neq j$},
 $$
 we have 
 $$
 \sum_{i=1}^N  {\mu}_{m_1}\left(\bigcup_{n\in \mathcal S_i} E_{n}\right)=
 {\mu}_{m_1}\left(\bigcup_{i=1}^N\bigcup_{n\in \mathcal S_i} E_{n}\right)=
 {\mu}_{m_1}\left(\bigcup_{n\in \mathcal S} E_{n}\right)\le {\mu}_{m_1}(S)\le C<N\eta,
 $$
 which, in turn, yields
 $$
 {\mu}_{m_1}\left(\bigcup_{n\in \mathcal S_{p}} E_{n}\right)<\eta
 $$
 for some $p\in\{1,\ldots, N\}$. Now let $m_2, m_3,\ldots$ be an increasing sequence of elements from the set $\mathcal S_p$. 

By induction we shall define the sequences $(m^{k}_i)$ for $k\ge 1$ in the following way. First set $m^1_i=m_i$ for $i=1, 2,\ldots$, where $(m_i)$ is an increasing sequence satisfying conditions (\ref{e2_28.04.16}) and (\ref{e3_28.04.16}) with $\eta=\eps/2$. Now if $(m^{k-1}_i)$ is given, by what we have already proven, we may find its subsequence 
$(m^{k}_i)$, $m^{k}_1>m^{k-1}_1$, satisfying conditions (\ref{e2_28.04.16}) and (\ref{e3_28.04.16}) with $\eta=\eps/2^{k}$. 

Now set $n_i:=m^i_1$ for $i=1, 2,\ldots$ and observe that
  $$
  {\mu}_{n_{i}}\left(\bigcup_{j\neq i} E_{n_{j}}\right)=\sum_{j<i}{\mu}_{n_{i}}
  (E_{n_{j}}) + {\mu}_{n_{i}}\left(\bigcup_{j> i} E_{n_{j}}\right)\le 
  \sum_{j<i}\eps/2^{j}+\eps/2^{i}<\eps.
  $$
The first term evaluation follows from (\ref{e3_28.04.16}), by the fact that $n_i$ is an element of the sequences $(m^j_n)$ for $j<i$. Similarly, the second term is evaluated by inequality (\ref{e2_28.04.16}).
\end{proof}

\subsection{Proof of Theorem \ref{thrm:weak-implies-strong positive}}

\begin{proof}{\it(Theorem \ref{thrm:weak-implies-strong positive}).}\ 
Let $(\mu_n)\subset \CM^+(S)$.  At the beginning we show that it is enough to prove the claim for $(\mu_n)\subset\CP(S)$.  In fact, from the assumption that $\lim_{n\to\infty}\pair{\mu_n}{f}$ exists for every $f\in\BL(S)$, in particular for $f\equiv 1$, we obtain that 
$\lim_{n\to\infty}\mu_n(S)$ also exists. Set $c=\lim_{n\to\infty}\mu_n(S)$ and observe that $c<\infty$, by the fact that 
$\sup_{n\ge 1}\|\mu_n\|_{TV}<\infty$.
If $c=0$, then we immediately see that $\mu\equiv 0$ fulfills the requirements of our theorem. On the other hand, if $c>0$, then, we can replace $\mu_n$ with $\tilde{\mu}_n:=\mu_n/\mu_n(S)$, which is a probability measure. If the theorem is proven to hold for $(\tilde{\mu}_n)$, then it holds for the $(\mu_n)$ as well. 

To prove the theorem it suffices to show that the family 
$\{\tilde{\mu}_n: n\ge 1\}$ is tight, by the following argument. By Prokhorov's Theorem there exists some measure $\mu_*\in\CP(S)$ and a subsequence $(n_m)$ such that $\tilde\mu_ {n_m}\to\mu_*$ weakly.  
Further, due to the fact that $\lim_{n\to\infty}\pair{\tilde{\mu}_n}{f}$ exists for any $f\in\BL(S)$, we obtain that 
$\lim_{n\to\infty}\pair{\tilde{\mu}_n}{f}=\pair{\mu_*}{f}$ for $f\in\BL(S)$. This in turn, together with the tightness of $\{\tilde{\mu}_n:n\geq 1\}$, implies that $\tilde{\mu}_n\to\mu_*$ $C_b(S)$-weakly, as $n\to\infty$. Indeed, the tightness allows restricting (approximately) to a compact subset $K$. The continuous bounded function on $S$, when restricted to $K$ can be approximated uniformly by a function in $\BL(K)$, since $\BL(K)\subset C(K)$ is $\|\cdot\|_\infty$-dense. The Metric Tietze Extension Theorem (cf. \cite{McS}) allows to extend the function in $\BL(K)$ to one in $\BL(S)$ without changing uniform norm and Lipschitz constant. The claim then follows.
The $C_b$-weak convergence of $\tilde{\mu}_n$ to $\mu_*$ is equivalent to $\|\tilde\mu_n-\mu_*\|_\BL^*\to 0$, as $n\to\infty$, because the latter norm metrises $C_b$-weak convergence on $\CM^+(S)$ (cf. \cite{Dudley:1966}, Theorem 6 and Theorem 8). For $\mu=c\mu_*$ we obtain that $\|\mu_n-\mu\|_\BL^*\to 0$, as $n\to\infty$.

To complete the proof, we have to prove the claim that the family $\{{\mu}_n: n\ge 1\}\subset \CP(S)$ is uniformly tight.
Assume, contrary to our claim, that it is not tight.
By Lemma \ref{lem:non-tight sequence prop}, passing to a subsequence if necessary, we may assume that there exists $\varepsilon>0$ 
and a sequence of compact sets $(K_{n})$  satisfying
 \begin{equation}\label{con1}
 {\mu}_{n}(K_{n})\ge\varepsilon\quad\text{ for every $n\ge 1$}
 \end{equation}
 and
  \begin{equation}\label{con2}
 \dist (K_n, K_m):=\min\{\rho(x, y): x\in K_{n}\,\,\text{and}\,\, y\in K_{m}\}>\varepsilon\quad\text{ for $m\neq n$}.
 \end{equation}
 
 From Lemma \ref{L1_28.04.16}, with $E_n:=K^{\eps/3}_n$, it follows that there exists a subsequence $(n_{i})$ such that for every $i\ge 1$ we have
 \begin{equation}\label{e00}
 {\mu}_{n_{i}}\left(\bigcup_{j\neq i} K_{n_{j}}^{\varepsilon/3}\right)<\varepsilon/2.
\end{equation}
 Note that $\dist (K_{n_{i}}^{\varepsilon/3}, K_{n_{j}}^{\varepsilon/3})>\varepsilon/3$ for $i\neq j$.

We define the function $f: X\to [0, 1]$ by the formula
 $$
 f(x)=\sum_{i=1}^{\infty} f_i(x),
 $$
 where $f_i$ are arbitrary Lipschitz functions with Lipschitz constant $3/\varepsilon$ satisfying
 $$
 f_{i |K_{n_{{2i}}}}=1\quad\text{and}\quad 0\le f_i\le {\bf 1}_{K_{n_{{2i}}}^{\varepsilon/3}}.
 $$
According to Lemma \ref{L1_28.04.16b}, $f\in\BL(S)$ (with $\|f\|_\infty\leq 1$ and $|f|_L\leq 6/\varepsilon$). 
 
 To finish the proof it is enough to observe that for every $i\ge 1$ we have
$$
 \langle {\mu}_{n_{2i}}, f\rangle
 =\sum_{j=1}^{\infty}\langle {\mu}_{n_{2i}}, f_j\rangle\geq {\mu}_{n_{2i}}\left(K_{n_{2i}}\right)
 \overset{(\ref{con1})}{\ge}\varepsilon
 $$
 and
 $$
 \langle {\mu}_{n_{2i+1}}, f\rangle 
 =\sum_{j=1}^{\infty}\langle {\mu}_{n_{2i+1}}, f_j\rangle\leq\sum_{j=1}^{\infty} {\mu}_{n_{2i+1}}\left(K_{n_{2j}}^{\varepsilon/3}\right){\leq}
{\mu}_{n_{2i+1}}\left(\bigcup_{j\neq 2i+1} K_{n_{j}}^{\varepsilon/3}\right)\overset{(\ref{e00})}<\varepsilon/2,
 $$
which contradicts the assumption that $\lim_{n\to\infty}\pair{{\mu}_n}{f}$ exists for every $f\in\BL(S)$. Thus the family $\{{\mu}_n: n\ge 1\}$ is tight and we are done.
\end{proof}

\begin{remark} 1.) An alternative proof is feasible, based upon the elaborate theory presented in \cite{Pachl:2013}. By taking $f=\ind$, one finds that $\sup_n\|\mu_n\|_\TV<\infty$. Since $\BL(S)$ is dense in the space $\mathrm{U}_b(S)$ of uniformly continuous bounded functions on $S$ for the supremum norm (cf. \cite{Dudley:1966}, Lemma 8), one finds that $\pair{\mu_n}{f}$ is Cauchy for every $f\in \mathrm{U}_b(S)$. According to \cite{Pachl:2013}, Theorem 5.45, there exists $\mu\in\CM(S)^+$ such that $\mu_n\to\mu$, $\mathrm{U}_b(S)$-weakly. Then \cite{Pachl:2013} Theorem 5.36 yields that $\|\mu_n-\mu\|_\BL^*\to 0$. \\
2.) In the proof we show that if $(\mu_n)$ is a sequence of positive Borel measures such that $\pair{\mu_n}{f}$ converges for every $f\in\BL(S)$, then $(\mu_n)$ is uniformly tight in $\CM^+(S)$. See \cite{Bogachev-II:2007}, Corollary 8.6.3, p. 204, for results in this direction when $\pair{\mu_n}{f}$ converges for every $f\in C_b(S)$. Under the additional condition that there exists $\mu_*\in\CM^+(S)$ such that $\pair{\mu_n}{f}\to\pair{\mu_*}{f}$ for every $f\in C_b(S)$, tightness results appeared already in e.g. \cite{LeCam:1957}, Theorem 4 for positive measures or \cite{Billingsley:1968}, Appendix III, Theorem 8 for probability measures. 
\end{remark}

\subsection{Proof of Theorem \ref{thrm:weak-implies-strong convergence}}


\begin{proof} {\it (Theorem \ref{thrm:weak-implies-strong convergence})}.\ 
Let $(\mu_n)\subset \CM(S)$ be signed measures such that $\sup_n\|\mu_n\|_\TV<\infty$. Denote by $\mu_n^{+}$ and $\mu_n^{-}$ the positive and negative part of $\mu_n$, $n\ge 1$, respectively.
We consider the following set
\begin{align*}
\begin{aligned}
\mathcal C:=\Big\{\bigl(\beta,(m_n),(\nu_{m_n}),(\vartheta_{m_n})\bigr):
\;&\beta\geq 0,\;
(m_n)\subset\NN \text{ -- an increasing sequence},\\
&\nu_{m_n},\vartheta_{m_n}\in \CP(S),
\;\lim_{n\to\infty}\|\nu_{m_n}-\vartheta_{m_n}\|_\BL^*=0 \\
&\text{and }\;\mu_{m_n}^+\geq \beta \nu_{m_n},\;\mu_{m_n}^-\geq \beta \vartheta_{m_n}\Big\}.
\end{aligned}
\end{align*}
We first observe that $\mathcal C\neq \emptyset$, which follows from the fact that $\bigl(0,(m_n),(\nu_{m_n}),(\vartheta_{m_n})\bigr)\in \mathcal C$ for arbitrary $(m_n)$ and $\nu_{m_n},\vartheta_{m_n}\in \CP(S)$ such that $\lim_{n\to\infty}\|\nu_{m_n}-\vartheta_{m_n}\|_\BL^*=0$. Moreover, since $\bar{c}:=\sup_{n\geq 1}\|\mu_n\|_{TV}<\infty$, we obtain that $0\leq\beta\leq \bar c$ for every $\beta$ for which there are some $(m_n)$ and $\nu_{m_n},\vartheta_{m_n}$ such that $\bigl(\beta,(m_n),(\nu_{m_n}),(\vartheta_{m_n})\bigr)\in \mathcal C$. We can therefore introduce
$$\alpha=\sup\bigl\{\beta: \bigl(\beta,(m_n),(\nu_{m_n}),(\vartheta_{m_n})\bigr)\in \mathcal C\bigr\}.$$

From the definition of $\alpha$ it follows that there exists a subsequence $(m_n)$ of positive integers and an increasing sequence $(\alpha_n)$ of nonnegative constants satisfying
$\lim_{n\to\infty}\alpha_n=\alpha$ and
$$
\mu_{m_n}^+\ge\alpha_n\nu_{m_n}\quad\text{and}\quad\mu_{m_n}^-\ge\alpha_n\vartheta_{m_n},
$$
where $\nu_{m_n},\vartheta_{m_n} \in\CP(S)$ are such that
 $\|\nu_{m_n} -\vartheta_{m_n}\|_\BL^*\to 0$ as $n\to\infty$.

To finish the proof it is enough to show that both the sequences $(\mu_{m_n}^+-\alpha_n\nu_{m_n})$ and $(\mu_{m_n}^- -\alpha_n\vartheta_{m_n})$
are tight. Indeed, then, by the Prokhorov Theorem (\cite{Bogachev-II:2007}, Theorem 8.6.2) there exists a subsequence $(m_{n_k})$ of $(m_n)$ and two measures $\mu^1$ and $\mu^2$ such that the sequences $(\mu_{m_{n_k}}^+-\alpha_{n_k}\nu_{m_{n_k}})$ and $(\mu_{m_{n_k}}^--\alpha_{n_k}\vartheta_{m_{n_k}})$ converge $C_b(S)$-weakly to a positive measure $\mu^1$ and $\mu^2$, respectively. Hence also in $\|\cdot\|_\BL^*$-norm, according to Theorem \ref{thrm:weak-implies-strong positive}. Consequently, $\|\mu_{m_{n_k}}-(\mu^1-\mu^2)\|_\BL^*\to 0$ as $k\to\infty$, by the fact that $\|\nu_{m_{n_k}} -\vartheta_{{m_{n_k}}}\|_\BL^*\to 0$ as $k\to\infty$. This will complete the proof of the theorem. Indeed, if we know that the sequence (and also any subsequence) has a convergent subsequence (in the dual bounded Lipschitz norm), then the sequence is also convergent due to the fact that the limit of all convergent subsequences is the same, by the assumption that $\lim_{n\to\infty}\pair{\mu_n}{f}$ exists for any $f\in \BL(S)$.

Assume now, contrary to our claim, that at least one of the families $(\mu_{m_n}^+-\alpha_n\nu_{m_n})$ or $(\mu_{m_n}^--\alpha_n\vartheta_{m_n})$, say 
the first one, is not tight.
By Lemma \ref{lem:non-tight sequence prop}, passing to a subsequence if necessary, we may assume that there exists $\varepsilon>0$ and a sequence of compact sets $(K_{n})$  satisfying
\begin{align}\label{eq:tilde-mu(K)>epsilon}
(\mu_{m_{n}}^+-\alpha_{n}\nu_{m_{n}})(K_{n})\ge \varepsilon
\end{align}
and
$$
\dist (K_{i},  K_{j})\ge \varepsilon\quad\text{for $i, j\in\mathbb N,\,\, i\neq j.$}
$$
Set
$$
\tilde\mu_n:=\mu_{m_{n}}^+-\alpha_{n}\nu_{m_n}\quad\text{and}\quad
\hat\mu_n:=\mu_{m_n}^- -\alpha_{n}\vartheta_{m_n}.
$$
\vskip 0.2cm

{\bf Claim:}\ {\it For any $0<\eta\leq 1$ there exist $j$, as large as we wish, and $\tau_j, \chi_j\in\CP(S)$ satisfying}
$$
\tilde\mu_j\ge(\varepsilon/2)\tau_j,\qquad \hat\mu_j\ge(\varepsilon/2)\chi_j\quad\mbox{{\it and}}\quad \|\tau_j-\chi_j\|_\BL^*\le \eta.
$$
\vskip 0.2cm

Consequently, there will exist a subsequence $(m_{j_n})$ such that
$$\mu_{m_{j_n}}^+=\alpha_{j_n}\nu_{m_{j_n}}+\tilde{\mu}_{j_n}\geq \alpha_{j_n}\nu_{m_{j_n}}+(\varepsilon/2)\tau_{j_n},$$
$$\mu_{m_{j_n}}^-\geq \alpha_{j_n}\vartheta_{m_{j_n}}+(\varepsilon/2)\chi_{j_n}\qquad \mbox{and}\qquad \|\tau_{j_n}-\chi_{j_n}\|_\BL^*\to 0\ \mbox{as}\ n\to\infty.$$
Now, if we define probability measures $\varrho_{m_{j_n}},\varsigma_{m_{j_n}}$ as follows
$$\varrho_{m_{j_n}}:= (\alpha_{j_n}\nu_{m_{j_n}}+(\varepsilon/2)\tau_{j_n})(\alpha_{j_n}+\varepsilon/2)^{-1},\qquad
\varsigma_{m_{j_n}}:= (\alpha_{j_n}\vartheta_{m_{j_n}}+(\varepsilon/2)\chi_{j_n})(\alpha_{j_n}+\varepsilon/2)^{-1},
$$
we will obtain 
$$
\mu_{m_{j_n}}^+\ge(\alpha_{j_n}+\varepsilon/2)\varrho_{m_{j_n}},\qquad\mu_{m_{j_n}}^-\ge(\alpha_{j_n}+\varepsilon/2)\varsigma_{m_{j_n}}
$$  
and
$\lim_{n\to\infty}\|\varrho_{m_{j_n}} -\varsigma_{m_{j_n}}\|_\BL^*=0$, 
which is impossible, because it contradicts the definition of $\alpha$, since $\lim_{n\to\infty} (\alpha_{j_n}+\varepsilon/2)>\alpha$.

Let us prove the claim. 
Set $\xi_n:=\tilde\mu_n+\hat\mu_n$ for $n\ge 1$ and let $C:=\sup_{n\ge 1}{\xi}_n(S)$. Observe that $C\le \sup_{n\ge 1}\|\mu_n\|_{\TV}<\infty$. 
Fix $0<\eta\leq 1$ and let $\kappa\in (0, \varepsilon/6)$ be such that $6\kappa(1/\varepsilon+2/\varepsilon^2)<\eta$. 
Lemma \ref{L1_28.04.16} yields an increasing sequence $(j_n)\subset\NN$ such that
\begin{equation}\label{e1_29.04.16}
\begin{aligned}
{\xi}_{j_n}\left(\bigcup_{l\neq n} K_{j_l}^{\varepsilon/3}\right)<\kappa/4
\end{aligned}
\end{equation}
and hence
$$
\tilde\mu_{j_n}\left(\bigcup_{l\neq n} K_{j_l}^{\varepsilon/3}\right)<\kappa/4\quad\text{and}\quad\hat\mu_{j_n}\left(\bigcup_{l\neq n} K_{j_l}^{\varepsilon/3}\right)<\kappa/4
$$
for all $n=1, 2,\ldots$. 

Choose  $N\ge 1$ such that $N\kappa/4 > C$ and set $W_{j_n}^p:=K_{j_n}^{p\varepsilon/(3N)}\setminus K_{j_n}^{(p-1)\varepsilon/(3N)}$ for $p=1,\ldots, N$. Observe that
$W_{j_n}^p\cap W_{j_n}^q=\emptyset$ for $p\neq q$. Since $\sum_{p=1}^N \xi_{j_n}(W_{j_n}^p)=\xi_{j_n}(\bigcup_{p=1}^N W_{j_n}^p)\le C$, $n\ge 1$, 
for every $n$ there exists $p_n\in\{1,\ldots, N\}$ such that
\begin{align}\label{eq:xi}
{\xi}_{j_n}\left(W_{j_n}^{p_n}\right)< \kappa/4.
\end{align}
Now we are in a position to define a sequence $(f_n)$ of functions from $S$ to $[-1, 1]$. The construction is as follows. For $n=2k+1$ for $k\ge 1$, we set $f_n\equiv 0$. On the other hand, to define functions $f_n$ for $n=2k$ we introduce the measures
$$
\tilde\mu_{j_n}'(\cdot)=\tilde\mu_{j_n}\left(\cdot\cap K_{j_n}^{(p_n-1)\varepsilon/(3N)}\right)
$$ 
and
$$
\hat\mu_{j_n}'(\cdot)=\hat\mu_{j_n}\left(\cdot\cap K_{j_n}^{(p_n-1)\varepsilon/(3N)}\right).
$$ 
Further, there exists a Lipschitz function $\tilde f_n: K_{j_n}^{(p_n-1)\varepsilon/(3N)}\to [{-1}, 1]$ with $|\tilde f_n|_L\le 1$ such that
$\pair{\tilde\mu_{j_n}^{\prime}-\hat\mu_{j_n}^{\prime}}{\tilde f_n}\ge \frac{1}{2} \|\tilde\mu^{\prime}_{j_n}-\hat\mu^{\prime}_{j_n}\|_\BL^*$. Let  
$f_n$ be a Lipschitz extension of the function $\tilde f_n$ to $S$ such that $f_n(x)=\tilde f_n(x)$ for $x\in K_{j_n}^{(p_n-1)\varepsilon/(3N)}$ and $f_n(x)=0$ for $x\notin K_{j_n}^{p_n\varepsilon/(3N)}$. We may assume that  $|f_n|_L\le 3N/\varepsilon$. The existence of the extension function follows from McShane's formula (see \cite{McS}).
Let $f=\sum_{k=1}^{\infty}f_{2n}$. 
Since $\dist (\supp f_i, \supp f_j)>\varepsilon/3$ for $i, j\ge 1$, $i\neq j$, 
$f$ is a bounded Lipschitz function, by Lemma \ref{L1_28.04.16b}. 

We show that $\langle \mu_{m_{j_i}}, f\rangle\le \kappa/2$  for $i=2k+1$. 
Indeed,  for $k$ sufficiently large we have
\begin{align*}
\left\langle \mu_{m_{j_{2k+1}}}, f\right\rangle&=\sum_{n=1}^{\infty}\left\langle \mu_{m_{j_{2k+1}}}, f_{2n}\right\rangle
\leq \sum_{n=1}^{\infty}\xi_{j_{2k+1}}\left(K_{j_{2n}}^{\varepsilon/3}\right)
+\alpha_{j_{2k+1}}\|\nu_{m_{j_{2k+1}}}-\vartheta_{m_{j_{2k+1}}}\|_\BL^*\\
&\leq \xi_{j_{2k+1}}\left(\bigcup_{l\neq 2k+1}K_{j_{l}}^{\varepsilon/3}\right)+\alpha_{j_{2k+1}}\|\nu_{m_{j_{2k+1}}}-\vartheta_{m_{j_{2k+1}}}\|_\BL^*\\
&\overset{(\ref{e1_29.04.16})}< {\kappa}/{4} + \alpha_{j_{2k+1}}\|\nu_{m_{j_{2k+1}}}-\vartheta_{m_{j_{2k+1}}}\|_\BL^*<{\kappa}/{2},
\end{align*}
by the properties of the measures $\nu_{m_{j_{2k+1}}}, \vartheta_{m_{j_{2k+1}}}$ and the definition of the functions $f_{2n}$. Therefore
\[
\lim_{i\to\infty} \pair{\mu_{m_{j_i}}}{f} = \lim_{k\to\infty}\pair{\mu_{m_{j_{2k+1}}}}{f}\leq \kappa/2,
\]
because we assume that the limit of $\pair{\mu_m}{f}$ exists.

On the other hand, for $i=2k$ we have 
\begin{align*}
\left\langle \mu_{m_{j_{2k}}}, f\right\rangle&=\sum_{n=1}^{\infty}\left\langle \mu_{m_{j_{2k}}}, f_{2n}\right\rangle
\ge -\sum_{n\neq k}^{\infty}\xi_{j_{2k}}\left(K_{j_{2n}}^{\varepsilon/3}\right)+\left\langle \mu_{m_{j_{2k}}}, f_{2n},\right\rangle\\
&\ge -\sum_{n\neq k}^{\infty}\xi_{j_{2k}}\left(K_{j_{2n}}^{\varepsilon/3}\right) - {\xi}_{j_{2k}}\left(W_{j_{2k}^{p_{2k}}}\right) +\left\langle\tilde\mu^{\prime}_{{j_{2k}}}-\hat\mu^{\prime}_{{j_{2k}}}, \tilde{f}_{2k}\right\rangle\\
&\ge -\kappa/{4}-\kappa/4 +\frac{1}{2} \|\tilde\mu^{\prime}_{j_{2k}}-\hat\mu^{\prime}_{j_{2k}}\|_\BL^*,
\end{align*}
by the fact that $\|f_{2n}\|_{\infty}\le 1$.
Since $\lim_{i\to\infty}\pair{\mu_{m_{j_i}}}{f}\le \kappa/2$, by the estimation obtained for $i=2k+1$ and the assumption that the limit exists, we have 
$$
-\kappa/{4}-\kappa/4 +\frac{1}{2} \|\tilde\mu^{\prime}_{j_{2k}}-\hat\mu^{\prime}_{j_{2k}}\|_\BL^*\le 3\kappa/4
$$
for $k$ sufficiently large and consequently
$$
\|\tilde\mu^{\prime}_{j_{2k}}-\hat\mu^{\prime}_{j_{2k}}\|_\BL^*\le 3\kappa
$$
for all $k$ sufficiently large. Thus
$$
\hat\mu^{\prime}_{j_{2k}}(S)\ge \tilde\mu^{\prime}_{j_{2k}}(S)-3\kappa\ge \varepsilon-\varepsilon/2=\varepsilon/2.
$$

Hence, for probability measures
$$
\tilde\nu_{j_{2k}}:=\tilde\mu^{\prime}_{j_{2k}}/\tilde\mu^{\prime}_{j_{2k}}(S)\quad\text{and}\quad
\hat\nu_{j_{2k}}:=\hat\mu^{\prime}_{j_{2k}}/\hat\mu^{\prime}_{j_{2k}}(S)
$$
we have for $k$ sufficiently large
$$
\tilde\mu_{j_{2k}}\ge \tilde\mu_{j_{2k}}^{\prime}\ge(\varepsilon/2)\tilde\nu_{j_{2k}}\quad\text{and}\quad \hat\mu_{j_{2k}}\ge \hat\mu_{j_{2k}}^{\prime}\ge(\varepsilon/2)\hat\nu_{j_{2k}}.
$$
Finally, observe that for $k$ sufficiently large,
\begin{align*}
\|\tilde\nu_{j_{2k}}-\hat\nu_{j_{2k}}\|_\BL^*&\le\|\tilde\mu_{j_{2k}}^{\prime}/\tilde\mu_{j_{2k}}^{\prime}(S)-\hat\mu_{j_{2k}}^{\prime}/\tilde\mu_{j_{2k}}^{\prime}(S)\|_\BL^*+\|\hat\mu_{j_{2k}}^{\prime}\|_\BL^*|1/\tilde\mu_{j_{2k}}^{\prime}(S)-1/\hat\mu_{j_{2k}}^{\prime}(S)|
\\
&\le(1/\tilde\mu_{j_{2k}}^{\prime}(S)) \|\tilde\mu_{j_{2k}}^{\prime}-\hat\mu_{j_{2k}}^{\prime}\|_\BL^*+ 1/(\tilde\mu_{j_{2k}}^{\prime}(S)\hat\mu_{j_{2k}}^{\prime}(S))|\tilde\mu_{j_{2k}}^{\prime}(S)-\hat\mu_{j_{2k}}^{\prime}(S)|
\\
&\le 6\kappa/\varepsilon+12\kappa/\varepsilon^2<\eta,
\end{align*}
by the fact that $\tilde\mu_{j_{2k}}^{\prime}(S), \hat\mu_{j_{2k}}^{\prime}(S)\ge\varepsilon/2$ and $|\tilde\mu_{j_{2k}}^{\prime}(S)-\hat\mu_{j_{2k}}^{\prime}(S)|\le\|\tilde\mu_{j_{2k}}^{\prime}-\hat\mu_{j_{2k}}^{\prime}\|_\BL^*\le 3\kappa$.
This completes the proof of the claim, hence the theorem. 
\end{proof}

\begin{remark}\label{rem:alternative approaches}
It is possible to prove Theorem \ref{thrm:weak-implies-strong convergence} by means of a reduction-to-$\ell^1$-trick, inspired by ideas in \cite{Pachl:1979,Pachl:2013}, cf. \cite{Hille:2016}. Another proof is feasible, starting from \cite{Pachl:1979}, Theorem 3.2, see \cite{Worm-thesis:2010}.
However, here we prefer to present an independent, `set-geometric' proof that is self-contained and founded on the well-established result for the case of positive measures, Theorem \ref{thrm:weak-implies-strong positive}.
\end{remark}

\section{Further consequence: an alternative proof for weak sequential completeness}
\label{sec:weak sequ completeness}

Theorem \ref{thrm:weak-implies-strong convergence} allows -- in the case of a Polish space -- to give an alternative proof of the well-known fact that $\CM(S)$ is $C_b(S)$-weakly sequentially complete, that goes back to Alexandrov \cite{Alexandrov:1943} and Varadarajan \cite{Varadarajan:1961}, see. e.g. \cite{Dudley:1966}, Theorem 1 or \cite{Bogachev-II:2007}, Theorem 8.7.1 for a more general topological setting. We include our proof based on Theorem \ref{thrm:weak-implies-strong convergence} here, because it employs an argument for reduction to functions in $\BL(S)$, which by itself is an interesting result. 

This reduction is based on the following observation. Let $\CDD_S$ be the set of all metrics on $S$ that metrize the topology of $S$ as a complete separable metric space. We need to stress the dependence of the space $\BL(S)$ on the chosen metric on $S$. So for $d\in\CDD_S$ we write $\BL(S,d)$ for the space of bounded Lipschitz functions on $(S,d)$. The key observation is, that
\begin{equation}\label{Cb as union of BLS}
C_b(S) = \bigcup_{d\in\CDD_S} \BL(S,d).
\end{equation}
In fact, fix $d_0\in\CDD_S$. If $f\in C_b(S)$, then 
\[
d_f(x,y) := d_0(x,y)\vee|f(x)-f(y)|
\]
is a metric on $S$ such that $d_f\in\CDD_S$ and $f\in\BL(S,d_f)$. Here $\vee$ denotes the maximum.

The precise statement we consider is the following:
\begin{theorem}[Weak sequential completeness]
Let $S$ be a Polish space. Let $(\mu_n)\subset\CM(S)$ be such that $\pair{\mu_n}{f}$ converges for every $f\in C_b(S)$. Then there exists $\mu_*\in\CM(S)$ such that $\pair{\mu_n}{f}\to\pair{\mu_*}{f}$ for every $f\in C_b(S)$.
\end{theorem} 

\begin{proof}
The norm of $\mu_n$ viewed as a continuous linear functional on $C_b(S)$ is its total variation norm. Hence, according to the Banach-Steinhaus Theorem, $\sup_{n\geq 1} \|\mu_n\|_\TV<\infty$. For any $d\in\CDD_S$, $\pair{\mu_n}{f}$ converges for every $f\in C_b(S)$, so in particular for every $f\in\BL(S,d)$. The sequence $(\mu_n)$ is bounded in total variation norm, so Theorem \ref{thrm:weak-implies-strong convergence} implies there exists $\mu_*^d\in\CM(S)$ such that $\pair{\mu_n}{f}\to\pair{\mu_*^d}{f}$ for every $f\in \BL(S,d)$. We proceed to show that the limit measure $\mu_*^d$ is independent of $d$.

Let $d'\in\CDD_S$. Put 
\[
\bar{d}(x,y):=d(x,y)\vee d'(x,y).
\]
Then $\bar{d}\in\CDD_S$ and $\BL(S,\bar{d})$ contains both $\BL(S,d)$ and $\BL(S,d')$. Let $C\subset S$ be closed. There exist sequences $(h_n)$ and
$(h'_n)$ in $\BL(S,d)$ and $\BL(S,d')$ respectively, such that $h_n\downarrow \ind_C$ and $h'_n\downarrow \ind_C$ pointwise. Both these sequences are in $\BL(S,\bar{d})$, so
\[
\mu_*^d(C) = \lim_{k\to\infty} \pair{\mu_*^d}{h_k} = \lim_{k\to\infty}\ \lim_{n\to\infty} \pair{\mu_n}{h_k}
= \lim_{k\to\infty} \pair{\mu_*^{\bar{d}}}{h_k} = \mu_*^{\bar{d}}(C).
\]
A similar argument applies to $\mu_*^{d'}$, using the sequence $(h'_n)$ in $\BL(S,d')$ instead of $(h_n)$. So $\mu_*^d$ and $\mu_*^{d'}$ (and $\mu_*^{\bar{d}}$) agree on the $\pi$-system consisting of closed sets, which generate the Borel $\sigma$-algebra. Hence these measures are equal on all Borel sets. That is, there exists $\mu_*\in\CM(S)$ such that $\pair{\mu_n}{f}\to\pair{\mu_*}{f}$ for every $f\in\BL(S,d)$ {\it for every} $d\in\CDD_S$. Thus for every $f\in C_b(S)$ in view of \eqref{Cb as union of BLS}.
\end{proof}

\vskip 0.2cm

{\bf Acknowledgement.}\ This paper resulted partially from yet unpublished results from Daniel Worm's completed PhD research project (cf. \cite{Worm-thesis:2010}), from a functional analysis seminar at Leiden University in the spring of 2016 and a PhD course in Milano in January 2014 by Sander Hille \cite{Hille:2016} (partially building on \cite{Pachl:2013}) on the topic of dynamical systems in spaces of measures, and discussions among the authors and participants in the seminar about various aspects of this topic. We thank Hanna Wojew\'odka for carefully reading the proof of the main theorem and Marcel de Jeu for discussions concerning the coincidence of topologies on sets of measures.

\end{document}